\documentclass[12pt,a4paper]{article}
\usepackage[margin=2cm]{geometry}
\usepackage{amsmath,amsthm,amsfonts,stmaryrd}
\usepackage[backend=biber]{biblatex}
\usepackage{hyperref}
\usepackage[capitalise]{cleveref}
\bibliography{uniform_stability}
\newcommand{\iprod}[1]{\langle#1\rangle}
\newcommand{\bigiprod}[1]{\bigl\langle#1\bigr\rangle}
\newcommand{\uX}{u_{\mathbb{X}}}
\newcommand{\fX}{f_{\mathbb{X}}}
\newcommand{\RX}{R_{\mathbb{X}}}
\newcommand{\thetaX}{\theta_{\mathbb{X}}}
\newcommand{\rhoX}{\rho_{\mathbb{X}}}
\newcommand{\uzeroX}{u_{0\mathbb{X}}}
\newcommand{\vB}{\vec{\mathcal{B}}}
\newcommand{\jump}[1]{\llbracket#1\rrbracket}
\newtheorem{theorem}{Theorem}[section]
\newtheorem{lemma}[theorem]{Lemma}

\newtheorem{remark}[theorem]{Remark}
\title{Uniform Stability for a Spatially-Discrete, \\
Subdiffusive Fokker--Planck Equation\thanks{The support of KFUPM through the
project No.~SB191003 is gratefully acknowledged.}}
\author{William McLean and Kassem Mustapha}
\date{December 23, 2020}
\begin{document}
\maketitle
\begin{abstract}
We prove stability estimates for the spatially discrete, Galerkin solution of a
fractional Fokker--Planck equation, improving on previous results in several 
respects.  Our main goal is to establish that the stability constants are 
bounded uniformly in the fractional diffusion exponent~$\alpha\in(0,1]$.  In 
addition, we account for the presence of an inhomogeneous term and show a 
stability estimate for the gradient of the Galerkin solution.  As a by-product,
the proofs of error bounds for a standard finite element approximation are 
simplified.
\end{abstract}
\section{Introduction}

We consider the stability of semidiscrete Galerkin methods for the
time-fractional Fokker--Planck 
equation~\cite{HenryLanglandsStraka2010,AngstmannEtAl2015}
\begin{equation}\label{eq: ibvp}
\begin{aligned}
\partial_tu-\nabla\cdot\bigl(\partial_t^{1-\alpha}\kappa\nabla u
	-\vec F\,\partial_t^{1-\alpha}u\bigr)&=g&
&\text{for $\vec x\in\Omega$ and $0<t\le T$,}\\
u&=u_0(\vec x)&&\text{for $\vec x\in\Omega$ when $t=0$,}\\
u&=0&&\text{for $\vec x\in\partial\Omega$ and $0<t\le T$.}
\end{aligned}
\end{equation}
Here, $\Omega$ is a bounded Lipschitz domain in $\mathbb{R}^d$ ($d\ge1$), the 
fractional exponent satisfies $0<\alpha\le1$ and the fractional time 
derivative is understood in the Riemann--Liouville sense: 
$\partial_t^{1-\alpha}=\partial_t\,\mathcal{I}^\alpha$ where the fractional 
integration operator~$\mathcal{I}^\alpha$ is defined as usual 
in~\eqref{eq: I mu} below. The diffusivity $\kappa\in L_\infty(\Omega)$ is 
assumed independent of time, positive and bounded below:
$\kappa(\vec x)\ge\kappa_{\min}>0$ for $\vec x\in\Omega$.
The forcing vector~$\vec F$ may depend both on $\vec x$~and $t$, and we assume 
that $\vec F$, $\partial_t\vec F$, $\nabla\cdot\vec F$~and 
$\nabla\cdot\partial_t\vec F$ are bounded on $\Omega\times[0,T]$.  Note 
that if $\alpha=1$ then $\partial_t^{1-\alpha}\phi=\phi$ so the governing 
equation in~\eqref{eq: ibvp} reduces to a classical Fokker--Planck equation.

If $\vec F$ is independent of~$t$, then by applying $\mathcal{I}^{1-\alpha}$ to 
both sides of the governing equation we find that \eqref{eq: ibvp} is 
equivalent to
\begin{equation}\label{eq: Caputo}
{}^{\mathrm{C}}\partial_t^\alpha u-\nabla\cdot\bigl(\kappa\nabla u-\vec Fu\bigr)
    =\mathcal{I}^{1-\alpha}g,
\end{equation}
where ${}^\mathrm{C}\partial_t^\alpha u=\mathcal{I}^{1-\alpha}\partial_tu$ is 
the Caputo fractional derivative of order~$\alpha$.  In this form, numerous
authors have studied the numerical solution of the problem, mostly for a 1D 
spatial domain~$\Omega=(0,L)$ and with $g\equiv0$.  For instance, 
Deng~\cite{Deng2007} considered the method of lines, Jiang and 
Xu~\cite{JiangXu2019} proposed a finite volume method,  Yang et 
al.~\cite{YangEtAl2018} a spectral collocation method, and Duong and 
Jin~\cite{DuongJin2020} a Wasserstein gradient flow formulation.

For both continuous and discrete solutions to fractional PDEs, it is natural to 
expect stability constants to remain bounded as~$\alpha\to1$ if the limiting
classical problem is stable.  In applications, the value of~$\alpha$ is 
typically estimated from measurements, and this process might be treated as 
an inverse problem. It would then be desirable that simulations of the 
forward problem are uniformly stable in~$\alpha$, particularly if the diffusion 
turns out to be classical ($\alpha=1$) or only slightly subdiffusive ($\alpha$ 
close to~$1$).  Perhaps for these reasons, interest in the question of 
$\alpha$-uniform stability and convergence seems to be growing.  Note that 
growth in the stability constant as~$\alpha\to0$ is of less concern, since 
very small values of~$\alpha$ have not been observed in real physical systems.

In the special case~$\vec F\equiv\vec 0$ of fractional diffusion, 
Chen and Stynes~\cite{ChenStynes2020} showed that as~$\alpha$ tends to~$1$
the solution of~\eqref{eq: Caputo} tends to the solution of the classical 
diffusion problem, uniformly in $x$~and $t$.  They also discussed several 
examples of numerical schemes for which the error analysis leads to constants 
that remain bounded as~$\alpha\to1$ (said to be $\alpha$-robust bounds), as 
well as several for which the constants blow up ($\alpha$-nonrobust bounds).  
Recent examples in the former category include Jin et 
al.~\cite[see Remark~4]{JinLiZhou2020}, Huang et 
al.~\cite[see Section~5]{HuangEtAl2020} and 
Mustapha~\cite[Lemma~3.1, Theorem~3.5]{Mustapha2020}.

We work with the weak solution~$u:(0,T]\to H^1_0(\Omega)$ of~\eqref{eq: ibvp}
characterized by
\begin{equation}\label{eq: u weak}
\iprod{u',v}+\iprod{\partial_t^{1-\alpha}\kappa\nabla u,\nabla v}
	-\iprod{\vec F\partial_t^{1-\alpha}u,\nabla v}=\iprod{g,v}
\quad\text{for $v\in H^1_0(\Omega)$}
\end{equation}
and $0<t\le T$, with $u(0)=u_0$, where $u'=\partial_t u$, 
$\iprod{u,v}=\int_\Omega uv$~and
$\iprod{\vec u,\vec v}=\int_\Omega \vec u\cdot\vec v$.  
Strictly speaking, to allow minimal assumptions on the regularity of the data 
$u_0$~and $g$, we define the solution~$u$ by requiring that it satisfy the 
time-integrated equation
\begin{equation}\label{eq: u integrated}
\iprod{u,v}+\iprod{\mathcal{I}^\alpha\kappa\nabla u,\nabla v}
-\bigiprod{\vB_1 u,\nabla v}=\iprod{f,v},
\end{equation}
where
\begin{equation}\label{eq: B1}
(\vB_1\phi)(t)=\int_0^t\bigl(\vec F\,\partial_t^{1-\alpha}\phi\bigr)(s)\,ds
\quad\text{and}\quad
f(t)=u_0+\int_0^t g(s)\,ds.
\end{equation}
In previous work, we have established that this problem is well-posed
\cite{LeMcLeanStynes2019,McLeanEtAl2019}.
 
For a fixed, finite dimensional subspace $\mathbb{X}\subseteq H^1_0(\Omega)$, 
the semidiscrete Galerkin solution~$\uX:[0,T]\to\mathbb{X}$ is given by
\begin{equation}\label{eq: integrated Galerkin}
\iprod{\uX,\chi}+\bigiprod{\mathcal{I}^\alpha\kappa\nabla\uX,\nabla\chi}
    -\bigiprod{\vB_1\uX,\nabla\chi}=\iprod{\fX,\chi}
\quad\text{for $\chi\in\mathbb{X}$},
\end{equation}
with $\fX(t)=\uzeroX+\int_0^t g(s)\,ds$ and with
$\uX(0)=\uzeroX\in\mathbb{X}$ a suitable approximation to~$u_0$.
Previously, we studied this problem in the particular case 
when~$\mathbb{X}=S_h$ is a space of continuous, piecewise-linear finite element 
functions corresponding to a conforming triangulation of~$\Omega$ with maximum 
element size~$h>0$.  We showed that the Galerkin finite element 
solution~$u_h(t)$ is stable in the norm of~$L_2(\Omega)$ when~$g(t)\equiv0$, 
satisfying the bound~\cite[Theorem~4.5]{LeMcLeanMustapha2018}
\begin{equation}\label{eq: uh stability g=0}
\|u_h(t)\|\le C_\alpha\|u_{0h}\|\quad\text{for $0\le t\le T$ and $0<\alpha<1$,}
\end{equation}
where $u_h(0)=u_{0h}\in S_h$ approximates~$u_0$. The method of proof relied on 
estimates for fractional integrals~\cite[Lemmas~3.2--3.4]{LeMcLeanMustapha2018} 
involving powers of~$(1-\alpha)^{-1}$, leading to a stability 
constant~$C_\alpha$ that blows up as~$\alpha\to1$. However, in the limiting 
case~$\alpha=1$ the semidiscrete finite element method is easily seen to be 
stable~\cite[Remark~4.7]{LeMcLeanMustapha2018}, that is, 
\eqref{eq: uh stability g=0} holds with $C_1<\infty$.  In the absence of 
forcing, that is, in the simple case~$\vec F\equiv\vec 0$ of fractional 
diffusion, the stability constant equals one: $\|u_h(t)\|\le\|u_{0h}\|$ 
for~$0<\alpha\le1$.

Our primary aim in what follows is to improve the results of our earlier 
paper~\cite{LeMcLeanMustapha2018} via a new analysis 
of~\eqref{eq: integrated Galerkin} that yields a uniform stability constant 
for~$0<\alpha\le1$, as well as allowing a non-zero source term~$g$.  Recently,
Huang et al.~\cite{HuangEtAl2020} have addressed the same question using a
different analysis that requires $1/2<\alpha\le1$ and $u_0\in H^1(\Omega)$,
with a stability constant that blows up as~$\alpha\to1/2$.

Throughout the paper, $C$ denotes a generic constant that may depend of $T$, 
$\Omega$, $\kappa$~and $\vec F$.  Dependence on any other parameters will be 
shown explicitly, and in particular, we write $C_\alpha$ to show that the 
constant may also depend on~$\alpha$.  
After citing some technical lemmas in \cref{sec: preliminaries}, we 
present the stability proof in \cref{sec: stability}, stating our main 
result as \cref{thm: uX stability}.  Using similar arguments, 
\cref{sec: gradient} establishes an estimate for the gradient of~$\uX$ 
in \cref{thm: grad uX}.  Combining these results gives, for 
$0\le t\le T$~and $0<\alpha\le1$,
\begin{equation}\label{eq: uX well-posed}
\begin{aligned}
\|\uX(t)\|+t^{\alpha/2}\|\nabla\uX(t)\|
    &\le C\biggl(\|\uzeroX\|+\int_0^t\|g(s)\|\,ds\biggr)
    +C\biggl(\frac{1}{t}\int_0^t\|sg(s)\|^2\,ds\biggr)^{1/2}\\
    &\le C\biggl(\|\uzeroX\|+t^{1/2}\int_0^t\|g(s)\|^2\,ds\biggr).
\end{aligned}
\end{equation}
(Here, the second bound follows from the first by 
\cref{lem: g alternative} with $\eta=1$.)  At the end of 
\cref{sec: gradient}, in \cref{remark: u stability}, we note that 
the exact solution~$u$ has the same uniform stability property as the 
semidiscrete Galerkin approximation~$\uX$, that is, \eqref{eq: uX well-posed} 
holds with $u$~and $u_0$ replacing $\uX$~and $\uzeroX$.  Also, in
\cref{remark: zero flux}, we discuss briefly the implications of 
enforcing a zero-flux boundary condition instead of the homogeneous Dirichlet 
one in problem~\eqref{eq: ibvp}.  \cref{sec: error} applies our new 
stability analysis to the piecewise linear Galerkin finite element 
solution~$u_h$, showing in \cref{thm: uh error} that 
$\|u_h-u\|=O(t^{-\alpha(2-r)/2}h^2)$~and 
$\|\nabla u_h-\nabla u\|=O(t^{-\alpha(2-r)/2}h)$.  These error bounds rely on
a regularity assumption involving the parameter~$r\in[0,2]$, which we justify 
in \cref{thm: u reg}. Finally, \cref{sec: DG} considers the time 
discretization of~\eqref{eq: ibvp} in the special case~$\vec F\equiv\vec 0$
of plain fractional diffusion using the discontinuous Galerkin method, proving 
a fully-discrete stability result in \cref{thm: DG}.

\section{Preliminaries}\label{sec: preliminaries}

This brief section introduces notations and gathers together results 
from the literature that we will use in our subsequent analysis.
Denote the fractional integral operator of order~$\mu>0$ by
\begin{equation}\label{eq: I mu}
(\mathcal{I}^\mu\phi)(t)=\int_0^t\omega_\mu(t-s)\phi(s)\,ds
\quad\text{for $t>0$, where $\omega_\mu(t)=\frac{t^{\mu-1}}{\Gamma(\mu)}$,}
\end{equation}
with~$\mathcal{I}^0\phi=\phi$, and observe that
$(\mathcal{I}^1\phi)(t)=\int_0^t\phi(s)\,ds$.  If we denote the Laplace 
transform of~$\phi$ by $\hat\phi(z)=\int_0^\infty e^{-zt}\phi(t)\,dt$ then
$(\widehat{\mathcal{I}^\mu\phi})(z)=\hat\omega_\mu(z)\hat\phi(z)$ and
$\hat\omega_\mu(z)=z^{-\mu}$.  Hence, assuming $\phi$ is real-valued with (say) 
compact support in~$[0,\infty)$, we find that
\begin{equation}\label{eq: I mu positive}
\int_0^t\bigiprod{\phi,\mathcal{I}^\mu\phi}\,ds
    =\frac{\pi\mu/2}{\pi}\int_0^\infty y^{-\mu}\|\hat\phi(iy)\|^2\,dy\ge0
\quad\text{if $0<\mu<1$.}
\end{equation}
Our analysis relies on 
properties of~$\mathcal{I}^\mu$ stated in the next three lemmas.

\begin{lemma}\label{lem: I nu I mu}
If $0\le\mu\le\nu\le1$, then for $t>0$~and 
$\phi\in L_2\bigl((0,t),L_2(\Omega)\bigr)$,
\[
\int_0^t\|(\mathcal{I}^\nu\phi)(s)\|^2\,ds\le2t^{2(\nu-\mu)}\int_0^t
    \|(\mathcal{I}^\mu\phi)(s)\|^2\,ds.
\]
\end{lemma}
\begin{proof}
See Le et al.~\cite[Lemma~3.1]{LeMcLeanMustapha2018}.
\end{proof}

\begin{lemma}\label{lem: I mu y}
If $0<\mu\le1$, then for $t>0$~and $\phi\in L_2\bigl((0,t),L_2(\Omega)\bigr)$,
\[
\int_0^t\|(\mathcal{I}^\mu\phi)(s)\|^2\,ds\le2\int_0^t\omega_\mu(t-s)
\int_0^s\bigiprod{\phi(q),(\mathcal{I}^\mu\phi)(q)}\,dq.
\]
\end{lemma}
\begin{proof}
See McLean et al.~\cite[Lemma~2.2]{McLeanEtAl2019}.
\end{proof}

\begin{lemma}\label{lem: phi(t)}
Let $0\le\mu<\nu\le1$. If $\phi:[0,T]\to L_2(\Omega)$ is continuous with 
$\phi(0)=0$, and if its restriction to~$(0,T]$ is differentiable 
with~$\|\phi'(t)\|\le Ct^{-\mu}$ for $0<t\le T$, then
\[
\|\phi(t)\|^2\le2\omega_{2-\nu}(t)\int_0^t
    \bigiprod{\phi'(s),(\mathcal{I}^\nu\phi')(s)}\,ds.
\]
\end{lemma}
\begin{proof}
See McLean et al.~\cite[Lemma~2.3]{McLeanEtAl2019}.
\end{proof}

We will also require the following fractional Gronwall inequality involving
the Mittag-Leffler 
function~$E_\mu(z)=\sum_{n=0}^\infty z^n/\Gamma(1+n\mu)$.

\begin{lemma}\label{lem: Gronwall}
Let $\mu>0$ and assume that $a$~and $b$ are non-negative and non-decreasing 
functions on the interval~$[0,T]$. If $y:[0,T]\to\mathbb{R}$ is an integrable 
function satisfying
\[
0\le y(t)\le a(t)+b(t)\int_0^t\omega_\mu(t-s)y(s)\,ds
        \quad\text{for $0\le t\le T$,}
\]
then
\[
y(t)\le a(t)E_\mu\bigl(b(t)\,t^\mu\bigr)\quad\text{for $0\le t\le T$.}
\]
\end{lemma}
\begin{proof}
See Dixon and McKee~\cite[Theorem~3.1]{DixonMcKee1986} and
Ye, Gao and Ding~\cite[Corollary~2]{YeGaoDing2007}.
\end{proof}

We recall the definition of the linear operator~$\vB_1$ in~\eqref{eq: B1},
and introduce two other linear operators $\mathcal{M}$~and $\vB_2$, defined by
\[
(\mathcal{M}\phi)(t)=t\,\phi(t)\quad\quad\text{and}
\qquad\vB_2\phi=\bigl(\mathcal{M}\vec{\mathcal{B}}_1\phi\bigr)'.
\]
With the help of the identity
\begin{equation}\label{eq: MI commutator}
\mathcal{M}\mathcal{I}^\alpha-\mathcal{I}^\alpha\mathcal{M}
    =\alpha\mathcal{I}^{\alpha+1}
\end{equation}
and using our assumption that $\vec F$~and $\vec F'$ are bounded, one can show 
the following technical estimates.

\begin{lemma}\label{lem: B1 B2}
Let $\mu>0$.  If $\phi:[0,T]\to L_2(\Omega)$ is continuous, and if its 
restriction to~$(0,T]$ is differentiable with $\|\phi'(t)\|\le Ct^{\mu-1}$ 
for $0<t\le T$, then
\[
\int_0^t\|\vB_1\phi\|^2\,ds\le 
C\int_0^t\|\mathcal{I}^\alpha\phi\|^2\,ds
\]
and
\[
\int_0^t\|\vB_2\phi\|^2\,ds\le C\int_0^t\bigl(
    \|\mathcal{I}^\alpha(\mathcal{M}\phi)'\|^2
    +\|\mathcal{I}^\alpha\phi\|^2\bigr)\,ds.
\]
\end{lemma}
\begin{proof}
The estimate for $\vB_1$ was proved by Le et 
al.~\cite[Lemma~4.1]{LeMcLeanMustapha2018}, who also proved the bound for 
$\vB_2$ (denoted there by $B_3$) but with the extra term
term~$C\int_0^t\|\mathcal{I}^\alpha\mathcal{M}\phi\|^2\,ds$.  However, we 
may omit this extra term because it is bounded by 
$Ct^2\int_0^t\|\mathcal{I}^\alpha\phi\|^2\,ds$, as one sees from 
\cref{lem: I nu I mu} and \eqref{eq: MI commutator}.
\end{proof}

\section{Stability analysis}\label{sec: stability}

To simplify the error analysis of \cref{sec: error} we will include an 
additional term on the right-hand side of~\eqref{eq: integrated Galerkin} and 
study the stability of~$\uX:[0,T]\to\mathbb{X}$ satisfying
\begin{equation}\label{eq: gen RHS}
\iprod{\uX,\chi}+\bigiprod{\kappa\mathcal{I}^\alpha\nabla\uX,\nabla\chi}
    -\bigiprod{\vB_1\uX,\nabla\chi}
    =\iprod{f_1,\chi}+\iprod{\vec f_2,\nabla\chi}
\quad\text{for $\chi\in\mathbb{X}$,}
\end{equation}
with $\uX(0)=\uzeroX$; the semidiscrete Galerkin solution is then given by the 
special case $f_1=\fX$~and $\vec f_2=\vec 0$.  Our goal is to bound 
$\|\uX(t)\|$ pointwise in~$t$, and for this purpose our overall strategy is to 
apply \cref{lem: phi(t)} with $\phi=\mathcal{M}\uX$, which in turn 
requires an estimate for~$\int_0^t\bigiprod{\mathcal{I}^\alpha(\mathcal{M}\uX)',
(\mathcal{M}\uX)'}\,ds$.  The technical details are worked out in 
\cref{lem: step 1,lem: step 2} below, and the stability 
estimate itself is then obtained in \cref{lem: uX pointwise} 
for~\eqref{eq: gen RHS}, and in \cref{thm: uX stability} for the 
original semidiscrete Galerkin problem~\eqref{eq: integrated Galerkin}.

Our method of proof begins by multiplying both sides of~\eqref{eq: gen RHS} 
by~$t$ and using the identity~\eqref{eq: MI commutator} to obtain
\[
\iprod{\mathcal{M}\uX,\chi}
+\bigiprod{\kappa\mathcal{I}^\alpha(\mathcal{M}\nabla\uX)
+\alpha\kappa\mathcal{I}^{\alpha+1}\nabla\uX
    -\mathcal{M}\vB_1\uX,\nabla\chi}
=\iprod{\mathcal{M}f_1,\chi}+\iprod{\mathcal{M}\vec f_2,\nabla\chi}.
\]
Differentiating this equation with respect to time then yields
\begin{multline}\label{eq: DM integrated Galerkin}
\bigiprod{(\mathcal{M}\uX)',\chi}
    +\bigiprod{\kappa\partial_t^{1-\alpha}(\mathcal{M}\nabla\uX)
    +\alpha\kappa\mathcal{I}^\alpha\nabla\uX-\vB_2\uX,\nabla\chi}\\
    =\iprod{(\mathcal{M}f_1)',\chi}+\iprod{(\mathcal{M}\vec f_2)',\nabla\chi}.
\end{multline}

\begin{lemma}\label{lem: step 1}
For $i\in\{0,1\}$ and $0\le t\le T$, the solution of~\eqref{eq: gen RHS}
satisfies
\[
\int_0^t\Bigl(\bigl\|\mathcal{I}^\alpha(\mathcal{M}^i\uX)\bigr\|^2+t^\alpha
\bigl\|\mathcal{I}^\alpha(\nabla\mathcal{M}^i\uX)\bigr\|^2\Bigr)\,ds
\le Ct^{2(\alpha+i)}\int_0^t\bigl(\|f_1\|^2+t^{-\alpha}\|\vec f_2\|^2\bigr)\,ds.
\]
\end{lemma}
\begin{proof}
Choose $\chi=(\mathcal{I}^\alpha\uX)(t)$ in~\eqref{eq: gen RHS} so that
\begin{multline*}
\iprod{\uX,\mathcal{I}^\alpha\uX}+\kappa_{\min}\|\mathcal{I}^\alpha\nabla\uX\|^2
    \le\bigiprod{\vec f_2+\vB_1\uX,\mathcal{I}^\alpha\nabla\uX}
        +\iprod{f_1,\mathcal{I}^\alpha\uX}\\
    \le\kappa_{\min}^{-1}\bigl(\|\vec f_2\|^2+\|\vB_1\uX\|^2\bigr)
    +\tfrac12\kappa_{\min}\|\mathcal{I}^\alpha\nabla\uX\|^2
    +\|f_1\|\|\mathcal{I}^\alpha\uX\|.
\end{multline*}
After cancelling $\tfrac12\kappa_{\min}\|\mathcal{I}^\alpha\nabla\uX\|^2$, 
integrating in time and applying \cref{lem: B1 B2}, we deduce that
\begin{multline}\label{eq: step 1.1}
\int_0^t\Bigl(\iprod{\uX,\mathcal{I}^\alpha\uX}
    +\tfrac12\kappa_{\min}\|\mathcal{I}^\alpha\nabla\uX\|^2\Bigr)\,ds\\
    \le C_0t^{-\alpha}\int_0^t\|\mathcal{I}^\alpha\uX\|^2\,ds
    +\int_0^t\bigl(t^\alpha\|f_1\|^2+\kappa_{\min}^{-1}\|\vec f_2\|^2\bigr)\,ds,
\end{multline}
where $C_0$ is a fixed constant depending on $T$, $\kappa_{\min}$
and $\vec F$. Apply $\mathcal{I}^\alpha$ to both sides of~\eqref{eq: gen RHS}
and choose~$\chi=\mathcal{I}^\alpha \uX(t)$ to obtain, for any~$\eta>0$,
\begin{multline*}
\|\mathcal{I}^\alpha\uX\|^2
+\iprod{\kappa\mathcal{I}^\alpha(\mathcal{I}^\alpha\nabla \uX),
	\mathcal{I}^\alpha\nabla \uX}
=\iprod{\mathcal{I}^\alpha(\vec f_2+\vB_1\uX),\mathcal{I}^\alpha\nabla\uX}
        +\iprod{\mathcal{I}^\alpha f_1,\mathcal{I}^\alpha\uX}\\
\le\eta\bigl(\|\mathcal{I}^\alpha\vec f_2\|^2
    +\|\mathcal{I}^\alpha(\vB_1\uX)\|^2\bigr)
+\tfrac12\eta^{-1}\|\mathcal{I}^\alpha\nabla\uX\|^2
+\tfrac12\|\mathcal{I}^\alpha f_1\|^2+\tfrac12\|\mathcal{I}^\alpha\uX\|^2.
\end{multline*}
Simplifying, integrating in time, and noting 
$\int_0^t\iprod{\kappa\mathcal{I}^\alpha(\mathcal{I}^\alpha\nabla\uX),
\mathcal{I}^\alpha\nabla\uX}\,ds\ge0$ by~\eqref{eq: I mu positive},
we observe that
\[
\int_0^t\|\mathcal{I}^\alpha\uX\|^2\,ds
	\le\int_0^t\|\mathcal{I}^\alpha f_1\|^2\,ds
+2\eta\int_0^t\bigl(\|\mathcal{I}^\alpha\vec f_2\|^2
    +\|\mathcal{I}^\alpha(\vB_1\uX)\|^2\bigr)\,ds
	+\frac{1}{\eta}\int_0^t\|\mathcal{I}^\alpha \nabla \uX\|^2\,ds.
\]
By \cref{lem: I nu I mu,lem: B1 B2},  
\[
\int_0^t \|\mathcal{I}^\alpha(\vB_1\uX)\|^2\,ds
    \le 2t^{2\alpha}\int_0^t\|\vB_1\uX\|^2\,ds 
\le Ct^{2\alpha}\int_0^t\|\mathcal{I}^\alpha\uX\|^2\,ds,
\]
and so, again with the help of \cref{lem: I nu I mu}, 
\begin{multline}\label{eq: step 1.2}
\int_0^t\|\mathcal{I}^\alpha \uX\|^2\,ds
    \le2t^{2\alpha}\int_0^t\bigl(\|f_1\|^2+2\eta\|\vec f_2\|^2\bigr)\,ds
    +C\eta t^{2\alpha}\int_0^t\|\mathcal{I}^\alpha\uX\|^2\,ds\\
    +\frac{1}{\eta}\int_0^t\|\mathcal{I}^\alpha\nabla\uX\|^2\,ds.
\end{multline}
However, by~\eqref{eq: step 1.1}, 
\[
\int_0^t\|\mathcal{I}^\alpha\nabla \uX\|^2\,ds
    \le\frac{2C_0}{\kappa_{\min}}\,t^{-\alpha}
    \int_0^t\|\mathcal{I}^\alpha\uX\|^2\,ds
    +C\int_0^t\bigl(t^\alpha\|f_1\|^2+\|\vec f_2\|^2\bigr)\,ds,
\]
and by \cref{lem: I mu y}, 
\begin{equation}\label{eq: step 1.2.1}
\int_0^t\|\mathcal{I}^\alpha\uX\|^2\,ds
    \le 2\int_0^t\omega_\alpha(t-s)\int_0^s
    \bigiprod{\uX(q),(\mathcal{I}^\alpha\uX)(q)}\,dq\,ds\,.
\end{equation}
Inserting these two estimates in the right-hand side of~\eqref{eq: step 1.2} 
and choosing $\eta=4C_0t^{-\alpha}/\kappa_{\min}$ yields 
\begin{equation}\label{eq: step 1.3}
\int_0^t\|\mathcal{I}^\alpha \uX\|^2\,ds
    \le Ct^\alpha\int_0^t\bigl(t^\alpha\|f_1\|^2+\|\vec f_2\|^2\bigr)\,ds
    +Ct^\alpha\int_0^t\omega_\alpha(t-s)\int_0^s
        \bigiprod{\uX,\mathcal{I}^\alpha\uX}\,dq\,ds.
\end{equation}
Let
$y(t)=\int_0^t\bigl(\iprod{\uX,\mathcal{I}^\alpha\uX}
+\|\mathcal{I}^\alpha\nabla \uX\|^2\bigr)\,ds$,
and deduce from \eqref{eq: step 1.1}~and \eqref{eq: step 1.3} that
\begin{equation}\label{eq: step 1.4}
y(t)\le C\int_0^t\bigl(t^\alpha\|f_1\|^2+\|\vec f_2\|^2\bigr)\,ds
        +C\int_0^t\omega_\alpha(t-s)y(s)\,ds.
\end{equation}
It follows by \cref{lem: Gronwall} that, for $0\le t\le T$,
\begin{equation}\label{eq: step 1.5}
y(t)\le CE_\alpha\bigl(Ct^\alpha\bigr)\int_0^t
    \bigl(t^\alpha\|f_1\|^2+\|\vec f_2\|^2\bigr)\,ds
    \le C\int_0^t\bigl(t^\alpha\|f_1\|^2+\|\vec f_2\|^2\bigr)\,ds.
\end{equation}
Together, \eqref{eq: step 1.2.1}~and \eqref{eq: step 1.5} imply
\[
\int_0^t\Bigl(\|\mathcal{I}^\alpha\uX\|^2
    +t^\alpha\|\mathcal{I}^\alpha\nabla\uX\|^2\Bigr)\,ds
    \le Ct^\alpha\int_0^t\bigl(t^\alpha\|f_1\|^2+\|\vec f_2\|^2\bigr)\,ds\\
    +Ct^\alpha\int_0^t\omega_\alpha(t-s)y(s)\,ds,
\]
and using \eqref{eq: step 1.5} a second time gives
\begin{multline*}
\int_0^t\omega_\alpha(t-s)y(s)\,ds\le C
    \biggl(\int_0^t\omega_\alpha(t-s)\,ds\biggr)
    \int_0^t\bigl(t^\alpha\|f_1\|^2+\|\vec f_2\|^2\bigr)\,ds\\
    \le Ct^\alpha\int_0^t\bigl(t^\alpha\|f_1\|^2+\|\vec f_2\|^2\bigr)\,ds,
\end{multline*}
completing the proof for the case~$i=0$.

The identity~\eqref{eq: MI commutator} implies that
\[
\|(\mathcal{I}^\alpha\mathcal{M}\uX)(t)\|^2
\le 2t^2\|(\mathcal{I}^\alpha\uX)(t)\|^2
+2\alpha^2\|(\mathcal{I}^{\alpha+1}\uX)(t)\|^2
\]
so, using \cref{lem: I nu I mu},
\begin{align*}
\int_0^t \|\mathcal{I}^\alpha\mathcal{M}\uX\|^2\,ds 
    &\le2t^2\int_0^t\|\mathcal{I}^\alpha\uX\|^2\,ds  
    +4\alpha^2 t^2\int_0^t\|\mathcal{I}^\alpha\uX\|^2\,ds
    \le 6t^2\int_0^t\|\mathcal{I}^\alpha\uX\|^2\,ds.
\end{align*}
Similarly, 
\[
\int_0^t\|\mathcal{I}^\alpha\mathcal{M}\nabla\uX\|^2\,ds 
\le 6t^2\int_0^t\|\mathcal{I}^\alpha \nabla\uX\|^2\,ds,
\]
so the case~$i=1$ follows from the already proven case~$i=0$.
\end{proof}

The next lemma makes use of the identity
\begin{equation}\label{eq:  D I alpha commutator}
(\partial_t^{1-\alpha}\phi)(t)=(\mathcal{I}^\alpha\phi)'(t)
    =\phi(0)\omega_\alpha(t)+(\mathcal{I}^\alpha\phi')(t).
\end{equation}

\begin{lemma}\label{lem: step 2}
For $0\le t\le T$, the solution of~\eqref{eq: gen RHS} satisfies
\begin{multline*}
\int_0^t\Bigl(\bigiprod{(\mathcal{M}\uX)',\mathcal{I}^\alpha(\mathcal{M}\uX)'}
        +\|\mathcal{I}^\alpha(\mathcal{M}\nabla\uX)'\|^2\Bigr)\,ds\\
        \le Ct^\alpha\int_0^t\bigl(\|f_1\|^2+\|(\mathcal{M}f_1)'\|^2\bigr)\,ds
        +C\int_0^t\bigl(\|\vec f_2\|^2+\|(\mathcal{M}\vec f_2)'\|^2\bigr)\,ds.
\end{multline*}
\end{lemma}
\begin{proof}
Rewriting \eqref{eq: DM integrated Galerkin} as
\begin{equation}\label{eq: DM rewrite}
\bigiprod{(\mathcal{M}\uX)',\chi}
    +\bigiprod{\kappa\partial_t^{1-\alpha}(\mathcal{M}\nabla\uX),\nabla\chi}
=\bigiprod{(\mathcal{M}\vec f_2)'+\vB_2\uX
-\alpha\kappa\mathcal{I}^\alpha\nabla\uX,\nabla\chi}
+\iprod{(\mathcal{M}f_1)',\chi},
\end{equation}
we note that the first term on the right is bounded by
\[
\tfrac12\kappa_{\min}\|\nabla\chi\|^2
    +\tfrac{3}{2}\kappa_{\min}^{-1}\bigl(\|(\mathcal{M}\vec f_2)'\|^2
    +\|\vB_2\uX\|^2+\alpha^2\|\kappa\mathcal{I}^\alpha\nabla\uX\|^2\bigr).
\]
Choose $\chi=\partial_t^{1-\alpha}(\mathcal{M}\uX)(t)
=\bigl(\mathcal{I}^\alpha(\mathcal{M}\uX)\bigr)'(t)$ and note that,
since $(\mathcal{M}\uX)(0)=0$, the identity \eqref{eq: D I alpha commutator} 
implies $\chi=\mathcal{I}^\alpha(\mathcal{M}\uX)'$ so
\begin{multline*}
\iprod{(\mathcal{M}\uX)',\mathcal{I}^\alpha(\mathcal{M}\uX)'}
        +\tfrac12\kappa_{\min}\|\mathcal{I}^\alpha(\mathcal{M}\nabla\uX)'\|^2\\
        \le\|(\mathcal{M}f_1)'\|\,\|\mathcal{I}^\alpha(\mathcal{M}\uX)'\|
    +C\|(\mathcal{M}\vec f_2)'\|^2+C\|\vB_2\uX\|^2
    +C\|\mathcal{I}^\alpha\nabla\uX\|^2.
\end{multline*}
Thus, by \cref{lem: B1 B2},
\begin{align*}
y(t)&\equiv\int_0^t\Bigl(
\bigiprod{(\mathcal{M}\uX)',\mathcal{I}^\alpha(\mathcal{M}\uX)'}
        +\|\mathcal{I}^\alpha(\mathcal{M}\nabla\uX)'\|^2\Bigr)\,ds\\
&\le C\int_0^t\bigl(
    t^\alpha\|(\mathcal{M}f_1)'\|^2+\|(\mathcal{M}\vec f_2)'\|^2\bigr)\,ds \\
&\qquad{}+C\int_0^t\bigl(\|\mathcal{I}^\alpha(\nabla\uX)\|^2
    +\|\mathcal{I}^\alpha\uX\|^2\bigr)\,ds
    +C_0t^{-\alpha}\int_0^t\|\mathcal{I}^\alpha(\mathcal{M}\uX)'\|^2\,ds.
\end{align*}
The second integral on the right is bounded
by~$Ct^\alpha\int_0^t\bigl(\|f_1\|^2+t^{-\alpha}\|\vec f_2\|^2\bigr)\,ds$ via 
\cref{lem: step 1}, giving
\begin{multline}\label{eq: step 2.1}
y(t)\le Ct^\alpha\int_0^t\bigl(\|f_1\|^2+\|(\mathcal{M}f_1)'\|^2\bigr)\,ds
    +C\int_0^t\bigl(\|\vec f_2\|^2+\|(\mathcal{M}\vec f_2)'\|^2\bigr)\,ds\\
    +C_0t^{-\alpha}\int_0^t\|\mathcal{I}^\alpha(\mathcal{M}\uX)'\|^2\,ds.
\end{multline}
Now apply $\mathcal{I}^\alpha$ to both sides 
of~\eqref{eq: DM rewrite}, again with
$\chi=\partial_t^{1-\alpha}(\mathcal{M}\uX)(t)
=\mathcal{I}^\alpha(\mathcal{M}\uX)'(t)$, to conclude that
\begin{align*}
\|\mathcal{I}^\alpha(\mathcal{M}\uX)'\|^2&+\bigiprod{\kappa
\mathcal{I}^\alpha\bigl(\mathcal{I}^\alpha(\mathcal{M}\nabla\uX)'\bigr),
\mathcal{I}^\alpha(\mathcal{M}\nabla\uX)'}\\
&\le\bigl(\|\mathcal{I}^\alpha(\mathcal{M}\vec f_2)'\|
    +\|\mathcal{I}^\alpha(\vB_2\uX)\|+\|\mathcal{I}^{2\alpha}\nabla\uX\|\bigr)
\|\mathcal{I}^\alpha(\mathcal{M}\nabla\uX)'\|\\
    &\qquad{}+\tfrac12\|\mathcal{I}^\alpha(\mathcal{M}f_1)'\|^2
    +\tfrac12\|\mathcal{I}^\alpha(\mathcal{M}\uX)'\|^2.
\end{align*}
After cancelling the last term on the right we have, for any~$\eta>0$,
\begin{multline*}
\|\mathcal{I}^\alpha(\mathcal{M}\uX)'\|^2+2\bigiprod{\kappa
\mathcal{I}^\alpha\bigl(\mathcal{I}^\alpha(\mathcal{M}\nabla\uX)'\bigr),
\mathcal{I}^\alpha(\mathcal{M}\nabla\uX)'}\\
\le3\eta\bigl(\|\mathcal{I}^\alpha(\mathcal{M}\vec f_2)'\|^2
    +\|\mathcal{I}^\alpha(\vB_2\uX)^2\|
    +\|\mathcal{I}^{2\alpha}\nabla\uX\|^2\bigr)
+\eta^{-1}\|\mathcal{I}^\alpha(\mathcal{M}\nabla\uX)'\|^2
    +\|\mathcal{I}^\alpha(\mathcal{M}f_1)'\|^2.
\end{multline*}
Since the integral over~$(0,t)$ of the second term on the left is non-negative, 
it follows using \cref{lem: I nu I mu} that
 \begin{multline*}
\int_0^t\|\mathcal{I}^\alpha(\mathcal{M}\uX)'\|^2\,ds
        \le 6\eta t^{2\alpha}\int_0^t\Bigl(\|(\mathcal{M}\vec f_2)'\|^2
    +\|\vB_2\uX\|^2 
    +\|\mathcal{I}^\alpha\nabla\uX\|^2\Bigr)\,ds\\
    +\eta^{-1}\int_0^t
        \|\mathcal{I}^\alpha(\mathcal{M}\nabla\uX)'\|^2\,ds
        +2t^{2\alpha}\int_0^t\|(\mathcal{M}f_1)'\|^2\,ds.
\end{multline*}
By \cref{lem: B1 B2,lem: step 1},
\begin{align*}
\int_0^t\Bigl(\|\vB_2\uX\|^2
    &+\|\mathcal{I}^\alpha(\nabla\uX)\|^2\Bigr)\,ds
\le Ct^\alpha\int_0^t\bigl(\|f_1\|^2+t^{-\alpha}\|\vec f_2\|^2\bigr)\,ds\\
    &\qquad{}+C\int_0^t\Bigl(\|\mathcal{I}^\alpha(\mathcal{M}\uX)'\|^2
    +\|\mathcal{I}^\alpha\uX\|^2\Bigr)\,ds\\
&\le C\int_0^t\|\mathcal{I}^\alpha(\mathcal{M}\uX)'\|^2\,ds
+C\bigl(t^\alpha+t^{2\alpha}\bigr)
    \int_0^t\bigl(\|f_1\|^2+t^{-\alpha}\|\vec f_2\|^2\bigr)\,ds,
\end{align*}
and consequently,
\begin{multline*}
C_0t^{-\alpha}\int_0^t\|\mathcal{I}^\alpha(\mathcal{M} \uX)'\|^2\,ds
    \le C t^\alpha\int_0^t\bigl(
        \eta t^\alpha\|f_1\|^2+\|(\mathcal{M}f_1)'\|^2\bigr)\,ds\\
+C\eta t^\alpha \int_0^t\bigl(
    \|\vec f_2\|^2+\|(\mathcal{M}\vec f_2)'\|^2\bigr)\,ds
    +C\eta t^\alpha\int_0^t\|\mathcal{I}^\alpha(\mathcal{M}\uX)'\|^2\,ds
+\frac{C_0t^{-\alpha}}{\eta}
        \int_0^t\|\mathcal{I}^\alpha(\mathcal{M} \nabla\uX)'\|^2\,ds.
\end{multline*}
Choosing $\eta=2C_0t^{-\alpha}$, we see from~\eqref{eq: step 2.1} that
\begin{multline*}
y(t)\le Ct^\alpha\int_0^t\bigl(\|f_1\|^2+\|(\mathcal{M}f_1)'\|^2\bigr)\,ds
    +C\int_0^t\bigl(\|\vec f_2\|^2+\|(\mathcal{M}\vec f_2)'\|^2\bigr)\,ds\\
    +C\int_0^t\|\mathcal{I}^\alpha(\mathcal{M}\uX)'\|^2\,ds.
\end{multline*}
The desired estimate follows after applying \cref{lem: I mu y} to bound 
the last integral on the right in terms of~$y$, and then applying 
\cref{lem: Gronwall}.
\end{proof}

\begin{lemma}\label{lem: uX pointwise}
For $0<t\le T$, the solution of~\eqref{eq: gen RHS} satisfies
\[
\|\uX(t)\|^2\le\frac{C}{t}\int_0^t\bigl(
    \|f_1\|^2+\|(\mathcal{M}f_1)'\|^2\bigr)\,ds
+\frac{C}{t^{1+\alpha}}\int_0^t\bigl(
    \|\vec f_2\|^2+\|(\mathcal{M}\vec f_2)'\|^2\bigr)\,ds.
\]
\end{lemma}
\begin{proof}
Apply \cref{lem: phi(t)}, with $\phi=\mathcal{M}\uX$, followed by 
\cref{lem: step 2}, to conclude that
\begin{align*}
t^2\|\uX(t)\|^2&=\|\mathcal{M}\uX(t)\|^2
    \le\frac{2t^{1-\alpha}}{\Gamma(2-\alpha)}
\int_0^t\bigiprod{\mathcal{I}^\alpha(\mathcal{M}\uX)',(\mathcal{M}\uX)'}\,ds\\
&\le Ct\int_0^t\bigl(\|f_1\|^2+\|(\mathcal{M}f_1)'\|^2\bigr)\,ds
    +Ct^{1-\alpha}\int_0^t\bigl(
    \|\vec f_2\|^2+\|(\mathcal{M}\vec f_2)'\|^2\bigr)\,ds,
\end{align*}
and then divide by~$t^2$. 
\end{proof}

\begin{theorem}\label{thm: uX stability}
The semidiscrete Galerkin solution, defined by~\eqref{eq: integrated Galerkin}, 
satisfies
\[
\|\uX(t)\|\le C\biggl(\|\uzeroX\|+\int_0^t\|g(s)\|\,ds\biggr)
    +C\biggl(\frac{1}{t}\int_0^t\|sg(s)\|^2\,ds\biggr)^{1/2}
\]
for $0<t\le T$, where the stability constant~$C$ depends on $T$, $\Omega$, 
$\kappa$~and $\vec F$, but not on $\alpha\in(0,1]$~or the subspace~$\mathbb{X}$.
\end{theorem}
\begin{proof}
Apply \cref{lem: uX pointwise} with $f_1=\fX$~and $\vec f_2=\vec 0$, 
noting that
\[
\frac{1}{t}\int_0^t\|\fX\|^2\,ds\le\max_{0\le s\le t}\|\fX(s)\|^2
\le\biggl(\|\uzeroX\|+\int_0^t\|g(s)\|\,ds\biggr)^2
\]
and $(\mathcal{M}\fX)'=\fX+\mathcal{M}g$.
\end{proof}

The terms in~$g$ from the above estimate can be bounded as follows.  In 
particular, by choosing $\eta=1$ we see that 
$\|\uX(t)\|\le C\|\uzeroX\|+Ct^{1/2}\|g\|_{L_2((0,T);L_2(\Omega))}$.

\begin{lemma}\label{lem: g alternative}
For $0<t\le T$~and $0<\eta\le1$, 
\[
\biggl(\int_0^t\|g(s)\|\,ds\biggr)^2+\frac{1}{t}\int_0^t\|sg(s)\|^2\,ds
    \le(1+\eta^{-1})t^\eta\int_0^t s^{1-\eta}\|g(s)\|^2\,ds.
\]
\end{lemma}
\begin{proof}
Using the Cauchy--Schwarz inequality,
\begin{align*}
\biggl(\int_0^t\|g(s)\|\,ds\biggr)^2&=
\biggl(\int_0^ts^{-(1-\eta)/2}s^{(1-\eta)/2}\|g(s)\|\,ds\biggr)^2\\
&\le\int_0^ts^{\eta-1}\,ds\int_0^ts^{1-\eta}\|g(s)\|^2\,ds
=\frac{t^\eta}{\eta}\int_0^t s^{1-\eta}\|g(s)\|^2\,ds,
\end{align*}
and furthermore,
\[
\frac{1}{t}\int_0^t\|sg(s)\|^2\,ds
    \le\int_0^t s\|g(s)\|^2\,ds
    \le t^\eta\int_0^ts^{1-\eta}\|g(s)\|^2\,ds.
\]
\end{proof}

\section{Gradient bounds}\label{sec: gradient}

A strategy similar to the one used in \cref{sec: stability} will allow 
us to bound~$\|\nabla\uX(t)\|$ pointwise in~$t$: we once again apply 
\cref{lem: phi(t)}, this time with $\phi=\mathcal{M}\nabla\uX$.  The key 
result is stated as \cref{lem: grad uX pointwise} for the generalized 
problem~\eqref{eq: gen RHS}, and as \cref{thm: grad uX} for the 
semidiscrete Galerkin equation~\eqref{eq: integrated Galerkin}.  The proofs 
rely on the following estimates; cf.~\cref{lem: B1 B2}.

\begin{lemma}\label{lem: div B}
Let $\mu>0$. If $\phi:[0,T]\to H^1(\Omega)$ is continuous, and if its 
restriction to~$(0,T]$ is differentiable with 
$\|\phi'(t)\|_{H^1(\Omega)}\le Ct^{\mu-1}$ for~$0<t\le T$, then
\[
\int_0^t\|\nabla\cdot(\vB_1\phi)\|^2\,ds\le C\int_0^t\bigl(
    \|\mathcal{I}^\alpha\phi\|^2+\|\mathcal{I}^\alpha\nabla\phi\|^2\bigr)\,ds
\]
and
\[
\int_0^t\|\nabla\cdot(\vB_2\phi)\|^2\,ds\le C\int_0^t\bigl(
    \|\mathcal{I}^\alpha\phi\|^2+\|\mathcal{I}^\alpha(\mathcal{M}\phi)'\|^2
    +\|\mathcal{I}^\alpha\nabla\phi\|^2
    +\|\mathcal{I}^\alpha(\mathcal{M}\nabla\phi)'\|^2\bigr)\,ds.
\]
\end{lemma}
\begin{proof}
Integration by parts gives
$(\vB_1\phi)(t)=\vec F\,(\mathcal{I}^\alpha\phi)(t)
-\int_0^t\vec F'\,\mathcal{I}^\alpha\phi\,ds$, so
\[
\nabla\cdot(\vB_1\phi)=(\nabla\cdot\vec F)(\mathcal{I}^\alpha\phi)
    +\vec F\cdot(\mathcal{I}^\alpha\nabla\phi)
    -\int_0^t\Bigl((\nabla\cdot\vec F)'(\mathcal{I}^\alpha\phi)
    +\vec F'\cdot(\mathcal{I}^\alpha\nabla\phi)\Bigr)\,ds,
\]
implying the first estimate.  Furthermore,
\[
\vB_2\phi=\vec F'\bigl(\mathcal{I}^\alpha\mathcal{M}\phi
    +\alpha\mathcal{I}^{\alpha+1}\phi\bigr)
    +\vec F\bigl(\mathcal{I}^\alpha(\mathcal{M}\phi)'
    +\alpha\mathcal{I}^\alpha\phi\bigr)
    -\mathcal{I}^1\bigl(\vec F'\,\mathcal{I}^\alpha\phi\bigr)
    -\mathcal{M}\vec F'\,\mathcal{I}^\alpha\phi,
\]
which implies the second estimate.
\end{proof}

The next result builds on the estimates of \cref{lem: step 1,lem: step 2}.

\begin{lemma}\label{lem: grad step}
The solution of~\eqref{eq: gen RHS} satisfies, for $0\le t\le T$,
\begin{multline*} 
\int_0^t\bigl(\|(\mathcal{M}\uX)'\|^2
+\bigiprod{\kappa\mathcal{I}^\alpha(\mathcal{M}\nabla\uX)',
    (\mathcal{M}\nabla\uX)'}\bigr)\,ds\le C\int_0^t\|\uX\|^2\,ds\\
+C\int_0^t\Bigl(\|f_1\|^2+\|(\mathcal{M}f_1)'\|^2
    +\|\vec f_2\|^2+\|(\mathcal{M}\vec f_2)'\|^2+\|\nabla\cdot\vec f_2\|^2
    +\|(\mathcal{M}\nabla\cdot\vec f_2)'\|^2\Bigr)\,ds.
\end{multline*}
\end{lemma}
\begin{proof}
Using the first Green identity, we deduce from~\eqref{eq: DM rewrite} that
\begin{multline*}
\bigiprod{(\mathcal{M}\uX)',\chi}
    +\bigiprod{\kappa(\mathcal{I}^\alpha\mathcal{M}\nabla\uX)',\nabla\chi}
    =\bigiprod{(\mathcal{M}f_1)'-(\mathcal{M}\nabla\cdot\vec f_2)'
        -\nabla\cdot(\vB_2\uX),\chi}\\
-\alpha\bigiprod{\kappa\mathcal{I}^\alpha\nabla\uX,\nabla\chi},
\end{multline*}
and from \eqref{eq: gen RHS} that
$\bigiprod{\kappa\mathcal{I}^\alpha\nabla\uX,\nabla\chi}
=\bigiprod{f_1-\nabla\cdot(\vB_1\uX)-\nabla\cdot\vec f_2-\uX,\chi}$, so
\begin{multline*}
\bigiprod{(\mathcal{M}\uX)',\chi}
    +\bigiprod{\kappa(\mathcal{I}^\alpha\mathcal{M}\nabla\uX)',\nabla\chi}
    =\bigiprod{f_3+\alpha\nabla\cdot(\vB_1\uX)-\nabla\cdot(\vB_2\uX)
    +\alpha\uX,\chi}\\
    \le\tfrac12\|\chi\|^2+2\bigl(\|f_3\|^2+\|\nabla\cdot(\vB_1\uX)\|^2
    +\|\nabla\cdot(\vB_2\uX)\|^2+\|\uX\|^2\bigr),
\end{multline*}
where $f_3=(\mathcal{M}f_1)'-\alpha f_1-(\mathcal{M}\nabla\cdot\vec f_2)'
+\alpha\nabla\cdot\vec f_2$. Choose $\chi=(\mathcal{M}\uX)'$, cancel the 
term $\tfrac12\|\chi\|^2$, and integrate in time to obtain
\begin{equation}\label{eq: grad step A}
\int_0^t\bigl(\|(\mathcal{M}\uX)'\|^2 
    +\bigiprod{\kappa\mathcal{I}^\alpha(\mathcal{M}\nabla\uX)',
    (\mathcal{M}\nabla\uX)'}
    \bigr)\,ds\le CJ(t)+C\int_0^t\|f_3\|^2\,ds
\end{equation}
where, by \cref{lem: div B},
\[
J(t)=\int_0^t\Bigl(\|\mathcal{I}^\alpha\uX\|^2 
    +\|\mathcal{I}^\alpha\nabla\uX\|^2
    +\|\mathcal{I}^\alpha(\mathcal{M}\uX)'\|^2
    +\|\mathcal{I}^\alpha(\mathcal{M}\nabla\uX)'\|^2+\|\uX\|^2\Bigr)\,ds.
\]
If we let $y(t)=\int_0^t\bigiprod{(\mathcal{M}\uX)', 
\mathcal{I}^\alpha(\mathcal{M}\uX)'}\,ds$ then, by \cref{lem: I mu y},
\[
\int_0^t\|\mathcal{I}^\alpha(\mathcal{M}\uX)'\|^2\,ds
    \le 2\int_0^t\omega_\alpha(t-s)y(s)\,ds
    \le 2\omega_{\alpha+1}(t)\max_{0\le s\le t}y(s),
\]
and so, using \cref{lem: step 1,lem: step 2}, 
\[
J(t)\le Ct^\alpha\int_0^t\bigl(\|f_1\|^2+\|(\mathcal{M}f_1)'\|^2\bigr)\,ds
    +C\int_0^t\bigl(\|\vec f_2\|^2+\|(\mathcal{M}\vec f_2)'\|^2\bigr)\,ds.
\]
Since
\[
\int_0^t\|f_3\|^2\,ds\le 4\int_0^t(\|f_1\|^2+\|(\mathcal{M}f_1)'\|^2
+\|\nabla\cdot\vec f_2\|^2+\|(\mathcal{M}\nabla\cdot\vec f_2)'\|^2)\,ds,
\]
the desired estimate now follows from~\eqref{eq: grad step A}.
\end{proof}

\begin{lemma}\label{lem: grad uX pointwise}
For $0<t\le T$,
\begin{multline*}
t^\alpha\|\nabla\uX(t)\|^2\le\frac{C}{t}\int_0^t\bigl(
    \|f_1\|^2+\|(\mathcal{M}f_1)'\|^2\bigr)\,ds\\
    +\frac{C}{t}\int_0^t\bigl(\|\vec f_2\|^2+\|(\mathcal{M}\vec f_2)'\|^2
    +\|\nabla\cdot\vec f_2\|^2 
    +\|(\mathcal{M}\nabla\cdot\vec f_2)'\|^2\bigr)\,ds.
\end{multline*}
\end{lemma}
\begin{proof}
Since $t^\alpha\|\nabla\uX(t)\|^2=t^{\alpha-2}\|\mathcal{M}\nabla\uX(t)\|^2$,
\cref{lem: phi(t),lem: grad step} imply that
\begin{multline*}
t^\alpha\|\nabla\uX(t)\|^2
    \le\frac{2\omega_{2-\alpha}(t)}{t^{2-\alpha}}\int_0^t
\bigiprod{\mathcal{I}^\alpha(\mathcal{M}\nabla\uX)',
    (\mathcal{M}\nabla\uX)'}\,ds\le\frac{C}{t}\int_0^t\|\uX\|^2\,ds\\
+\frac{C}{t}\int_0^t\Bigl(\|f_1\|^2+\|(\mathcal{M}f_1)'\|^2
    +\|\vec f_2\|^2+\|(\mathcal{M}\vec f_2)'\|^2+\|\nabla\cdot\vec f_2\|^2
    +\|(\mathcal{M}\nabla\cdot\vec f_2)'\|^2\Bigr)\,ds,
\end{multline*}
and it suffices to estimate $\int_0^t\|\uX\|^2\,ds$.  Choose 
$\chi=\uX(t)$ in~\eqref{eq: gen RHS} and use the first Green identity to deduce 
that
\begin{align*}
\|\uX\|^2+\bigiprod{\kappa\mathcal{I}^\alpha\nabla\uX,\nabla\uX}
    &=\bigiprod{f_1-\nabla\cdot\vec f_2-\nabla\cdot\vB_1\uX,\uX}\\
    &\le\tfrac12\|\uX\|^2+\tfrac32\bigl(\|f_1\|^2+\|\nabla\cdot\vec f_2\|^2
    +\|\nabla\cdot\vB_1\uX\|^2\bigr).
\end{align*}
After cancelling~$\tfrac12\|\uX\|^2$, integrating in time and using
\eqref{eq: I mu positive}, we have
\[
\int_0^t\|\uX\|^2\,ds
    \le3\int_0^t\bigl(\|f_1\|^2+\|\nabla\cdot\vec f_2\|^2\bigr)\,ds
    +3\int_0^t\|\nabla\cdot(\vB_1\uX)\|^2\,ds,
\]
and by \cref{lem: div B,lem: step 1},
\[
\int_0^t\|\nabla\cdot(\vB_1\uX)\|^2\,ds\le C\int_0^t\bigl(
    \|\mathcal{I}^\alpha\uX\|^2+\|\mathcal{I}^\alpha\nabla\uX\|^2\bigr)\,ds
    \le Ct^\alpha\int_0^t\bigl(\|f_1\|^2+t^{-\alpha}\|\vec f_2\|^2\bigr)\,ds,
\]
which completes the proof.
\end{proof}

The main result for this section now follows easily; once again, the terms 
in~$g$ may be estimated using \cref{lem: g alternative}.

\begin{theorem}\label{thm: grad uX}
The semidiscrete Galerkin solution, defined by~\eqref{eq: integrated Galerkin},
satisfies
\[
t^{\alpha/2}\|\nabla\uX(t)\|
\le C\biggl(\|\uzeroX\|+\int_0^t\|g(s)\|\,ds\biggr)+C\biggl(\frac{1}{t}
\int_0^t\|sg(s)\|^2\,ds\biggr)^{1/2}
\]
for $0<t\le T$, where the constant~$C$ depends on $T$, $\Omega$, $\kappa$~and 
$\vec F$, but not on $\alpha\in(0,1]$~or the subspace~$\mathbb{X}$.
\end{theorem}
\begin{proof}
Choose $f_1=\fX$~and $f_2=0$ in \cref{lem: grad uX pointwise}, and 
estimate $\fX$ in terms of $\uzeroX$~and $g$ using the same steps as in the 
proof of \cref{thm: uX stability}.
\end{proof}

\begin{remark}\label{remark: u stability}
The uniform stability estimate~\eqref{eq: uX well-posed} for the 
semidiscrete Galerkin solution carries over to the weak solution~$u$ of the 
continuous problem~\eqref{eq: ibvp}, that is,
\[
\|u(t)\|+t^{\alpha/2}\|\nabla u(t)\|
    \le C\biggl(\|u_0\|+\int_0^t\|g(s)\|\,ds\biggr)
    +C\biggl(\frac{1}{t}\int_0^t\|sg(s)\|^2\,ds\biggr)^{1/2}.
\]
Essentially, it suffices to repeat the steps in an earlier stability 
proof~\cite[Theorem~4.1]{McLeanEtAl2019} using \eqref{eq: uX well-posed} as a 
drop-in replacement for an estimate~\cite[Theorem~3.3]{McLeanEtAl2019} in
which the stability constant was dependent on~$\alpha$.
\end{remark}

\begin{remark}\label{remark: zero flux}
By introducing a flux vector~$\vec Qu=-\partial_t^{1-\alpha}\kappa\nabla u
+\vec F\partial_t^{1-\alpha}u$ we can write the fractional Fokker--Planck 
equation~\eqref{eq: ibvp} as a conservation law: 
$\partial_tu+\nabla\cdot\vec Qu=g$. It is then natural to consider a zero-flux 
boundary condition,
\begin{equation}\label{eq: zero flux}
\vec n\cdot\vec Qu=0\quad\text{for $\vec x\in\partial\Omega$ and $0<t\le T$,}
\end{equation}
where $\vec n$ denotes the outward unit normal to~$\Omega$.  (Notice that this 
boundary condition is non-local in time.) In this case, the weak 
solution~$u:(0,T]\to H^1(\Omega)$ is again characterized 
by~\eqref{eq: u weak}, and hence satisfies \eqref{eq: u integrated}, but with 
the test functions~$v$ now taken from the larger space~$H^1(\Omega)$.  We can
then choose a finite dimensional subspace~$\mathbb{X}\subseteq H^1(\Omega)$ and 
again define the Galerkin solution~$\uX:[0,T]\to\mathbb{X}$ 
by~\eqref{eq: integrated Galerkin}.  The analysis of 
\cref{sec: stability} goes through with no change, and in 
particular~$\uX$ is again stable in $L_2(\Omega)$. However, the first step in 
the proof of \cref{lem: grad step} fails because boundary terms are 
introduced if one integrates by parts in space, so our analysis no longer 
yields a bound for~$t^{\alpha/2}\|\nabla\uX(t)\|$.
\end{remark}

\section{Error estimates}\label{sec: error}

We now decompose the error in the semidiscrete Galerkin solution as
\[
\uX-u=\thetaX-\rhoX\quad\text{where}\quad\thetaX=\uX-\RX u\quad\text{and}\quad
\rhoX=u-\RX u,
\]
and where $\RX$ denotes the Ritz projector for the (stationary) elliptic problem
\begin{equation}\label{eq: elliptic problem}
-\nabla\cdot(\kappa\nabla v)+v=g
\quad\text{in $\Omega$, with $v=0$ on~$\partial\Omega$.}
\end{equation}
Thus, $\RX:H^1_0(\Omega)\to\mathbb{X}$ satisfies
\begin{equation}\label{eq: RX}
\iprod{\kappa\nabla\RX v,\nabla\chi}+\iprod{\RX v,\chi}
    =\iprod{\kappa\nabla v,\nabla\chi}+\iprod{v,\chi}
    \quad\text{for $v\in H^1_0(\Omega)$ and $\chi\in\mathbb{X}$,}
\end{equation}
or in other words, $\RX:v\mapsto v_{\mathbb{X}}$ where 
$v_{\mathbb{X}}\in\mathbb{X}$ is the Galerkin solution of the elliptic 
problem~\eqref{eq: elliptic problem}.  Note that, by including the lower-order 
term~$v$, the Ritz projector~$\RX:H^1(\Omega)\to\mathbb{X}$ would also be 
well-defined for the zero-flux boundary condition~\eqref{eq: zero flux}. 

It follows from \eqref{eq: u integrated}, \eqref{eq: integrated Galerkin}~and 
\eqref{eq: RX} that $\thetaX:[0,T]\to\mathbb{X}$ satisfies
\begin{equation}\label{eq: thetaX}
\iprod{\thetaX(t),\chi}
+\bigiprod{\mathcal{I}^\alpha(\kappa\nabla\thetaX)
-\vB_1\thetaX,\nabla\chi}=\iprod{f_1,\chi}+\iprod{\vec f_2,\nabla\chi}
\quad\text{for $\chi\in\mathbb{X}$,}
\end{equation}
where
\begin{equation}\label{eq: f1 f2}
f_1=(\uzeroX-P_{\mathbb{X}}u_0)+(\rhoX-\mathcal{I}^\alpha\rhoX),\qquad
\vec f_2=-\vB_1\rhoX, 
\end{equation}
and $P_{\mathbb{X}}:L_2(\Omega)\to\mathbb{X}$ is the orthoprojector 
given by $\iprod{P_{\mathbb{X}}v,\chi}=\iprod{v,\chi}$
for $v\in L_2(\Omega)$ and $\chi\in\mathbb{X}$.  If $u_0\in H^1_0(\Omega)$ so 
that $\RX u_0$ exists, then $\iprod{f_1,\chi}=\iprod{\tilde f_1,\chi}$ where
\begin{equation}\label{eq: f1 tilde}
\tilde f_1=(\uzeroX-\RX u_0)
    +\bigl(\rhoX-\rhoX(0)\bigr)-\mathcal{I}^\alpha\rhoX.
\end{equation}

We estimate $\thetaX$ in terms of~$\rhoX$ and the error in the discrete 
initial data~$\uzeroX$, as follows.

\begin{lemma}\label{lem: thetaX pointwise}
For $0<t\le T$,
\[
\|\thetaX(t)\|^2\le C\|\uzeroX-P_{\mathbb{X}}u_0\|^2
    +\frac{C}{t}\int_0^t\bigl(\|\rho_X\|^2+s^2\|\rhoX'\|^2\bigr)\,ds.
\]
\end{lemma}
\begin{proof}
Noting that \eqref{eq: thetaX} has the same form as~\eqref{eq: gen RHS}, with
$\thetaX$ playing the role of~$\uX$, we may apply 
\cref{lem: uX pointwise} and conclude that
\[
\|\thetaX(t)\|^2\le\frac{C}{t}\int_0^t\bigl(
    \|f_1\|^2+s^2\|f_1'\|^2\bigr)\,ds
    +\frac{C}{t^{1+\alpha}}\int_0^t\bigl(
    \|\vec f_2\|^2+s^2\|\vec f_2'\|^2\bigr)\,ds.
\]
Since $f_1'=\rhoX'-\partial_t^{1-\alpha}\rhoX$, we find with the help 
of \cref{lem: I nu I mu} that
\[
\frac{C}{t}\int_0^t\bigl(\|f_1\|^2+s^2\|f_1'\|^2\bigr)\,ds
    \le C\|\uzeroX-P_{\mathbb{X}}u_0\|^2
    +\frac{C}{t}\int_0^t\bigl(
    \|\rho_X\|^2+s^2\|\rhoX'\|^2+s^2\|\partial_s^{1-\alpha}\rhoX\|^2\bigr)\,ds.
\]
Using the identity~\eqref{eq: MI commutator} and noting that 
$(\mathcal{M}\rhoX)(0)=0$,
\begin{align*}
s\partial_s^{1-\alpha}\rhoX&=s\partial_s\mathcal{I}^\alpha\rhoX 
    =\partial_s(\mathcal{M}\mathcal{I}^\alpha\rhoX)-\mathcal{I}^\alpha\rhoX
    =\partial_s\bigl(\mathcal{I}^\alpha\mathcal{M}\rhoX
        +\alpha\mathcal{I}^{\alpha+1}\rhoX\bigr)-\mathcal{I}^\alpha\rhoX\\
    &=\mathcal{I}^\alpha(\mathcal{M}\rhoX)'+(\alpha-1)\mathcal{I}^\alpha\rhoX
    =\mathcal{I}^\alpha\bigl(\mathcal{M}\rhoX'+\alpha\rhoX\bigr),
\end{align*}
so by \cref{lem: I nu I mu},
\[
\int_0^t s^2\|\partial_s^{1-\alpha}\rhoX\|^2\,ds
    \le2t^{2\alpha}\int_0^t\bigl\|s\rhoX'+\alpha\rhoX\bigr\|^2\,ds
    \le 4t^{2\alpha}\int_0^t\bigl(\|\rhoX\|^2+s^2\|\rhoX'\|^2\bigr)\,ds,
\]
and hence
\begin{equation}\label{eq: f1 rho}
\frac{C}{t}\int_0^t\bigl(\|f_1\|^2+s^2\|f_1'\|^2\bigr)\,ds
    \le C\|\uzeroX-P_{\mathbb{X}}u_0\|^2\\
    +\frac{C}{t}\int_0^t\bigl(\|\rho_X\|^2+s^2\|\rhoX'\|^2\bigr)\,ds.
\end{equation}

Recalling~\eqref{eq: B1}, we have 
$f_2'(t)=-(\vec F\partial_t^{1-\alpha}\rhoX)(t)$ and therefore
by \cref{lem: B1 B2},
\[
\frac{C}{t^{1+\alpha}}\int_0^t\bigl(\|\vec f_2\|^2+s^2\|\vec f_2'\|^2\bigr)\,ds
    \le\frac{C}{t^{1+\alpha}}\int_0^t\|\mathcal{I}^\alpha\rhoX\|^2\,ds
    +\frac{C}{t^{1+\alpha}}\int_0^ts^2\|\partial_s^{1-\alpha}\rhoX\|^2\,ds,
\]
which is bounded by the second term on the right-hand side 
of~\eqref{eq: f1 rho}.
\end{proof}

Two similar bounds hold for $\nabla\thetaX$, but now
involving also $\nabla\rhoX$~and $\nabla\rhoX'$.

\begin{lemma}\label{lem: nabla thetaX}
For $0<t\le T$,
\begin{multline*}
t^\alpha\|\nabla\thetaX(t)\|^2\le C\|\uzeroX-P_{\mathbb{X}}u_0\|^2
    +\frac{C}{t}\int_0^t\bigl(\|\rho_X\|^2+s^2\|\rhoX'\|^2\bigr)\,ds\\
+Ct^{2\alpha-1}\int_0^t\bigl(\|\nabla\rhoX\|^2+s^2\|\nabla\rhoX'\|^2\bigr)\,ds.
\end{multline*}
If $u_0\in H^1_0(\Omega)$, then we also have the alternative bound
\begin{multline*}
t^\alpha\|\nabla\thetaX(t)\|^2\le C\bigl\|\uzeroX-\RX u_0\bigr\|^2
    +\frac{C}{t}\int_0^t\bigl(\|\rho_X-\rho_X(0)\|^2+s^2\|\rhoX'\|^2\bigr)\,ds\\
+Ct^{2\alpha-1}\int_0^t\bigl(\|\rho_X\|^2+\|\nabla\rhoX\|^2
    +s^2\|\nabla\rhoX'\|^2\bigr)\,ds.
\end{multline*}
\end{lemma}
\begin{proof}
With $f_1$~and $\vec f_2$ given by~\eqref{eq: f1 f2}, we apply 
\cref{lem: grad uX pointwise} to~\eqref{eq: thetaX} and bound
$t^\alpha\|\nabla\thetaX(t)\|^2$ by
\[
\frac{C}{t}\int_0^t\Bigl(
    \|f_1\|^2+\|(\mathcal{M}f_1)'\|^2+\|\vec f_2\|^2
    +\|(\mathcal{M}\vec f_2)'\|^2+\|\nabla\cdot\vec f_2\|^2
    +\|(\mathcal{M}\nabla\cdot\vec f_2)'\|^2\Bigr)\,ds.
\]
The terms in~$f_1$ can be bounded as in~\eqref{eq: f1 rho}, and since
$\vec f_2=-\vB_1\rhoX$~and $(\mathcal{M}\vec f_2)'=-\vB_2\rhoX$ we see from 
\cref{lem: B1 B2} followed by \cref{lem: div B} and then 
\cref{lem: I nu I mu} that 
\begin{align*}
\int_0^t\bigl(&\|\vec f_2\|^2+\|(\mathcal{M}\vec f_2)'\|^2
    +\|\nabla\cdot\vec f_2\|^2
    +\|(\mathcal{M}\nabla\cdot\vec f_2)'\|^2\bigr)\,ds\\
&\le C\int_0^t\Bigl(
    \|\mathcal{I}^\alpha\rhoX\|^2+\|\mathcal{I}^\alpha(\mathcal{M}\rhoX)'\|^2
    +\|\mathcal{I}^\alpha\nabla\rhoX\|^2
    +\|\mathcal{I}^\alpha(\mathcal{M}\nabla\rhoX)'\|^2\Bigr)\,ds\\
&\le Ct^{2\alpha}\int_0^t\Bigl(\|\rhoX\|^2+s^2\|\rhoX'\|^2+\|\nabla\rhoX\|^2
    +s^2\|\nabla\rhoX'\|^2\Bigr)\,ds,
\end{align*}
which completes the proof of the first bound.  If $u_0\in H^1_0(\Omega)$ then 
we can replace $f_1$ with~$\tilde f_1$ from~\eqref{eq: f1 tilde}, and since 
$\tilde f_1'=f_1'$ the second bound follows easily via the arguments leading 
to~\eqref{eq: f1 rho} (with $\RX$ replacing $P_{\mathbb{X}}$).
\end{proof}

To obtain more explicit error bounds we will use the regularity properties 
stated in the next theorem. The seminorm~$|\cdot|_r$ and norm~$\|\cdot\|_r$ in 
the (fractional-order) Sobolev space~$\dot H^r(\Omega)$ is defined in the usual 
way~\cite{Thomee2006} via the Dirichlet eigenfunctions of the Laplacian 
on~$\Omega$, and this spatial domain is assumed convex to ensure 
$H^2$-regularity for the elliptic problem. The proof relies on 
results~\cite[Lemma~2,Theorems 11--13]{McLeanEtAl2020} involving constants that 
blow up as~$\alpha\to1$.  Nevertheless, the estimates 
\eqref{eq: u reg 1}--\eqref{eq: u reg 3} hold in the limiting case~$\alpha=1$, 
when the problem reduces to the classical Fokker--Planck PDE; see 
Thom\'ee~\cite[Lemmas 3.2~and 4.4]{Thomee2006} for a proof if~$M=0$.

\begin{theorem}\label{thm: u reg}
Assume that $\Omega$ is convex, $0<\alpha<1$, $0\le r\le2$ and $\eta>0$.  If 
$u_0\in\dot H^r(\Omega)$~and if $g:(0,T]\to L_2(\Omega)$ is continuously
differentiable with $\|g(t)\|+t\|g'(t)\|\le Mt^{\eta-1}$, 
then the weak solution of~\eqref{eq: ibvp} satisfies, for $0<t\le T$,
\begin{equation}\label{eq: u reg 1}
\|u(t)\|_1 \le C_{\alpha,\eta}
    \bigl(\|u_0\|_r\,t^{-\alpha(1-r)/2}
        +Mt^{\eta-\alpha/2}\bigr)\quad\text{if $r\le1$,}
\end{equation}
and
\begin{equation}\label{eq: u reg 2}
t^{-\alpha/2}\|u(t)-u_0\|_1
    \le C_{\alpha,\eta}\bigl(\|u_0\|_r\,t^{-\alpha(2-r)/2}
    +Mt^{\eta-\alpha}\bigr)\quad\text{if $r\ge1$,}
\end{equation}
and
\begin{equation}\label{eq: u reg 3}
t^{1-\alpha/2}\|u'(t)\|_1+\|u(t)\|_2+t\|u'(t)\|_2
    \le C_{\alpha,\eta}
    \bigl(\|u_0\|_r\,t^{-\alpha(2-r)/2}+Mt^{\eta-\alpha}\bigr).
\end{equation}
\end{theorem}
\begin{proof}
We showed~\cite[Theorem~11]{McLeanEtAl2020} that
\[
\|u(t)\|_1 
    \le C_{\alpha,\eta}\bigl(\|u_0\|t^{-\alpha/2}+Mt^{\eta-\alpha/2}\bigr)
\]
and~\cite[Theorem~12]{McLeanEtAl2020} that
\[
\|u(t)-u_0\|_1+t\|u'(t)\|_1\le 
    C_{\alpha,\eta}\bigl(\|u_0\|_1+Mt^{\eta-\alpha/2}\bigr).
\]
Hence, $\|u(t)\|_1\le C\bigl(\|u_0\|_1+Mt^{\eta-\alpha/2}\bigr)$ and 
\eqref{eq: u reg 1} follows by interpolation.  The 
estimates \eqref{eq: u reg 2}~and \eqref{eq: u reg 3} were proved 
already~\cite[Theorems 12~and 13]{McLeanEtAl2020}.
\end{proof}

Now consider the concrete example in which $\mathbb{X}=S_h$ is the usual
continuous piecewise-linear finite element space for a triangulation 
of~$\Omega\subseteq\mathbb{R}^d$ with maximum element diameter~$h$, and
use a subscript~$h$ instead of~$\mathbb{X}$, writing $u_h$, $\theta_h$, 
$\rho_h$ etc.  The error in the Ritz projection satisfies
\begin{equation}\label{eq: Ritz error}
\|\rho_h(t)\|+h\|\nabla\rho_h(t)\|\le Ch^r|u(t)|_r\quad\text{for $r\in\{1,2\}$,}
\end{equation}
allowing us to prove the following error bounds for $u_h$~and $\nabla u_h$.  
Notice that if $0<\alpha<1/2$, then the restriction~$\alpha(2-r)<1$ is satisfied 
for all~$r\in[0,2]$, but if $1/2\le\alpha<1$ (and hence 
$0\le2-\alpha^{-1}<1$) then we are limited to~$r\in(2-\alpha^{-1},2]$.

\begin{theorem}\label{thm: uh error}
Let $0\le r\le2$ with $\alpha(2-r)<1$, and assume that the assumptions of 
\cref{thm: u reg} are satisfied with~$\eta\ge\alpha r/2$. Then, the 
semidiscrete finite element solution~$u_h$ satisfies the error bound
\[
\|u_h(t)-u(t)\|\le 
C\|u_{0h}-P_hu_0\|+C_\alpha h^2\,
\frac{t^{-\alpha(2-r)/2}}{\sqrt{1-\alpha(2-r)}}\,(\|u_0\|_r+M),
\]
and the gradient of~$u_h$ satisfies
\[
\|\nabla u_h(t)-\nabla u(t)\|
\le Ct^{-\alpha/2}\|u_{0h}-Q_{r,h}u_0\|
+C_\alpha h\,\frac{t^{-\alpha(2-r)/2})}{\sqrt{1-\alpha(2-r)}}\,(\|u_0\|_r+M),
\]
where $Q_{r,h}$ is either $P_h$ if $r\le1$, or else $R_h$ if $r\ge1$.
\end{theorem}
\begin{proof}
For brevity, let $K_r=(\|u_0\|_r+M)^2$. Using \eqref{eq: Ritz error} followed
by~\eqref{eq: u reg 3}, we have
\[
\|\rho_h(s)\|^2+s^2\|\rho_h'(s)\|^2\le Ch^4\bigl(
|u(s)|_2^2+s^2|u'(s)|_2^2\bigr)\le C_\alpha K_r h^4s^{-\alpha(2-r)},
\]
so, because of the assumption~$\alpha(2-r)<1$, 
\[
\frac{1}{t}\int_0^t\bigl(\|\rho_h(s)\|^2+s^2\|\rho_h'(s)\|^2\bigr)\,ds
\le C_\alpha K_rh^4\,\frac{t^{-\alpha(2-r)}}{1-\alpha(2-r)}.
\]
Since $\|u_h-u\|\le\|\theta_h\|+\|\rho_h\|$ and 
$\|\rho_h(t)\|^2\le C_\alpha K_rt^{-\alpha(2-r)}h^4$, the error bound for~$u_h$ 
follows by \cref{lem: thetaX pointwise}.

To estimate the error in~$\nabla u_h$, we apply \eqref{eq: Ritz error}~and
\eqref{eq: u reg 1} to obtain
\[
\|\rho_h(s)\|^2+s^2\|\rho_h'(s)\|^2\le Ch^2\bigl(
|u(s)|_1^2+s^2|u'(s)|_1^2\bigr)\le C_\alpha K_rh^2s^{-\alpha(1-r)} 
\quad\text{if $r\le1$,}
\]
and \eqref{eq: u reg 3} to obtain
\[
\|\nabla\rho_h(s)\|^2+s^2\|\nabla\rho_h'(s)\|^2\le Ch^2\bigl(
    |u(s)|_2^2+s^2|u'(s)|_2^2\bigr)
\le C_\alpha K_rh^2 s^{-\alpha(2-r)},
\]
so
\[
\frac{1}{t}\int_0^t\bigl(\|\rho_h(s)\|^2+s^2\|\rho_h'(s)\|^2\bigr)\,ds
    \le t^\alpha\,C_\alpha K_rh^2\,\frac{t^{-\alpha(2-r)})}{1-\alpha(1-r)}
    \quad\text{if $r\le1$,}
\]
and
\begin{equation}\label{eq: nabla rho_h}
t^{2\alpha-1}\int_0^t\bigl(\|\nabla\rho_h\|^2+s^2\|\nabla\rho_h'\|^2\bigr)\,ds
    \le\frac{C_\alpha K_rt^{\alpha r}h^2}{1-\alpha(2-r)}
    =t^{2\alpha}\,C_\alpha K_rh^2\,\frac{t^{-\alpha(2-r)}}{1-\alpha(2-r)}.
\end{equation}
Since 
$\|\nabla u_h(t)-\nabla u(t)\|\le\|\nabla\theta_h(t)\|+\|\nabla\rho_h(t)\|$,
the first estimate of \cref{lem: nabla thetaX} implies that
the error bound for~$\nabla u_h$ holds for the case~$r\le1$.

If $r\ge1$, then we see using \eqref{eq: u reg 1}--\eqref{eq: Ritz error} that
\[
\|\rho_h(s)-\rho_h(0)\|^2+s^2\|\rho_h'(s)\|^2
    \le Ch^2\bigl(\|u(s)-u(0)\|_1^2+s^2\|u'(s)\|_1^2\bigr)
    \le s^\alpha C_\alpha K_rh^2s^{-\alpha(2-r)}
\]
and $\|\rho_h(s)\|^2\le Ch^2\|u(s)\|_1^2
\le C_\alpha K_1h^2\le C_\alpha K_rh^2$, so
\[
\frac{1}{t}\int_0^t\bigl(\|\rho_h-\rho_h(0)\|^2+s^2\|\rho_h'(s)\|^2\bigr)\,ds
    +t^{2\alpha-1}\int_0^t\|\rho_h\|^2\,ds
    \le t^\alpha C_\alpha K_rh^2\,\frac{t^{-\alpha(2-r)}}{1-\alpha(2-r)}.
\]
Hence, using the second estimate of \cref{lem: nabla thetaX}
and \eqref{eq: nabla rho_h}, the error bound for~$\nabla u_h$ follows also for 
the case~$r\ge1$.
\end{proof}

\begin{remark}
If $r=2$ then by choosing $u_{0h}=R_hu_0$ we obtain an error bound that is 
uniform in time:
\[
\|u_h(t)-u(t)\|+h\|\nabla u_h(t)-\nabla u(t)\|\le C_\alpha h^2(\|u_0\|_2+M)
\quad\text{for $0<t\le T$.}
\]
\end{remark}

\begin{remark}
As a consequence of \cref{remark: zero flux}, if the zero-flux boundary 
condition~\eqref{eq: zero flux} is imposed then the proof of the error 
bound for~$u_h$ in \cref{thm: uh error} remains valid, but not that of
the error bound for~$\nabla u_h$.
\end{remark}

\section{Discontinuous Galerkin time stepping when $\vec F\equiv\vec0$}
\label{sec: DG}

We briefly consider a fully-discrete scheme for the fractional diffusion 
equation, that is, for the problem~\eqref{eq: ibvp} in the 
case~$\vec F\equiv\vec0$.  For time levels $0=t_0<t_1<\cdots<t_N=T$ we denote 
the $n$th time interval by~$I_n=(t_{n-1},t_n)$ and the $n$th step size 
by~$k_n=t_n-t_{n-1}$.  We choose an integer~$p_n\ge0$ for each time 
interval~$I_n$, and define the vector space~$\mathcal{W}$ consisting of all 
functions $X:\bigcup_{n=1}^NI_n\to\mathbb{X}$ such that the 
restriction~$X|_{I_n}$ is a polynomial in~$t$ of degree at most~$p_n$ with 
coefficients in~$\mathbb{X}$. For any $X\in\mathcal{W}$, write
\[
X^n_+=\lim_{\epsilon\downarrow0}X(t_n+\epsilon),\qquad
X^n_-=\lim_{\epsilon\downarrow0}X(t_n-\epsilon),\qquad
\jump{X}^n=X^n_+-X^n_-,
\]
then the discontinuous Galerkin (DG) solution~$U\in\mathcal{W}$ is defined by 
requiring that~\cite{Mustapha2015}
\begin{equation}\label{eq: DG}
\bigiprod{\jump{U}^n,X^{n-1}_+}+\int_{I_n}\iprod{\partial_tU,X}\,dt
    +\int_{I_n}\iprod{\partial_t^{1-\alpha}\kappa\nabla U,\nabla X}\,dt
    =\int_{I_n}\iprod{g,X}\,dt
\end{equation}
for $X\in\mathcal{W}$ and $1\le n\le N$, with $U^0_-=\uzeroX$ (so that
$\jump{U}^0=U^0_+-\uzeroX$).  To state a stability estimate for this scheme, 
let $C_\Omega$ denote the constant arising in the Poincar\'e inequality 
for~$\Omega$,
\begin{equation}\label{eq: Poincare}
\|v\|^2\le C_\Omega\|\nabla v\|^2\quad\text{for $v\in H^1_0(\Omega)$,}
\end{equation}
and define $\Psi:(0,1]\to\mathbb{R}$ by
\[
\Psi(\alpha)=\frac{1}{\pi^{1-\alpha}}\,
    \frac{(2-\alpha)^{2-\alpha}}{(1-\alpha)^{1-\alpha}}\,
    \frac{1}{\sin(\tfrac12\pi\alpha)}\quad\text{for $0<\alpha<1$.}
\]
Notice that $\Psi(1)=\lim_{\alpha\to1}\Psi(\alpha)=1$ but 
$\Psi(\alpha)\sim8\pi^{-2}\alpha^{-1}$ blows up as~$\alpha\to0$.  We will use 
the inequality~\cite[Theorem~A.1]{McLean2012}
\begin{equation}\label{eq: Psi bound}
\int_0^T\bigiprod{\partial_t^{1-\alpha}v,v}\,dt
    \ge\frac{T^{1-\alpha}}{\Psi(\alpha)}\int_0^T\|v\|^2\,dt.
\end{equation}

\begin{theorem}\label{thm: DG}
If $0<\alpha\le1$, then the DG solution of the fractional diffusion problem 
satisfies
\[
\|U^n_-\|^2+\sum_{j=1}^{n-1}\|\jump{U}^j\|^2
    +\int_0^{t_n}\bigiprod{\partial_t^{1-\alpha}\kappa\nabla U,\nabla U}\,dt
    \le\|\uzeroX\|^2+\frac{C_\Omega\Psi(\alpha)}{\kappa_{\min}t_n^{1-\alpha}}
    \int_0^{t_n}\|g(t)\|^2\,dt
\]
for $1\le n\le N$.
\end{theorem}
\begin{proof}
Let $B(U,X)$ denote the bilinear form
\[
\iprod{U^0_+,X^0_+}+\sum_{n=1}^{N-1}\iprod{\jump{U}^n,X^n_+}
    +\sum_{n=1}^N\int_{I_n}\iprod{\partial_tU,X}\,dt
    +\int_0^T\bigiprod{\partial_t^{1-\alpha}\kappa\nabla U,\nabla X}\,dt,
\]
and observe that the time-stepping equations~\eqref{eq: DG} are equivalent to
\[
B(U,X)=\iprod{\uzeroX,X^0_+}+\int_0^T\iprod{g,X}\,dt\quad
    \text{for $X\in\mathcal{W}$.}
\]
Taking $X=U$, we find by arguing as in the proof of 
Mustapha~\cite[Theorem~1]{Mustapha2015} that
\[
B(U,U)=\frac{1}{2}\biggl(\|U^0_+\|^2+\|U^N_-\|^2
    +\sum_{n=1}^{N-1}\|\jump{U}^n\|^2\biggr)
    +\int_0^T\bigiprod{\partial_t^{1-\alpha}\kappa\nabla U,\nabla X}\,dt,
\]
and so
\[
\|U^0_+\|^2+\|U^N_-\|^2+\sum_{n=1}^{N-1}\|\jump{U}^n\|^2
    +2\int_0^T\bigiprod{\partial_t^{1-\alpha}\kappa\nabla U,\nabla U}\,dt\\
    =2\iprod{\uzeroX,U^0_+}+2\int_0^T\iprod{g,U}\,dt.
\]
For any constant~$M>0$,
\[
2\iprod{\uzeroX,U^0_+}+2\int_0^T\iprod{g,U}\,dt
    \le\|U^0_+\|^2+\|\uzeroX\|^2
    +M\int_0^T\|g\|^2\,dt+\frac{1}{M}\int_0^T\|U\|^2\,dt,
\]
and using \eqref{eq: Poincare}~and \eqref{eq: Psi bound},
\[
\frac{1}{M}\int_0^T\|U\|^2\,dt\le\frac{C_\Omega}{M}\int_0^T\|\nabla U\|^2\,dt
    \le\frac{C_\Omega\Psi(\alpha)}{M\kappa_{\min}T^{1-\alpha}}
    \int_0^T\bigiprod{\partial_t^{1-\alpha}\kappa\nabla U,\nabla U}\,dt.
\]
Choosing $M=C_\Omega\Psi(\alpha)/(\kappa_{\min}T^{1-\alpha})$ implies the 
estimate in the case~$n=N$, which completes the proof since $T=t_N$.
\end{proof}

\begin{remark}
If $p_n=0$ then we have $\|U(t)\|\le\|U^n_-\|$ for~$t\in I_N$, and likewise if 
$p_n=1$ then $\|U(t)\|\le\max(\|U^n_-\|,\|U^{n-1}_+\|)\le\|U^n_-\|+\|U^{n-1}_-\|
+\|\jump{U}^{n-1}\|$ for~$t\in I_n$.  Thus, for the piecewise 
constant~\cite{McLeanMustapha2009} and 
piecewise linear~\cite{MustaphaMcLean2013} DG schemes we can prove stability 
in~$L_\infty(L_2)$, uniformly for~$\alpha$ bounded away from zero.
\end{remark}

\begin{remark}
For the solution~$u$ of the continuous fractional diffusion problem we have the 
analogous stability property
\[
\|u(t)\|^2+\int_0^t\bigiprod{\partial_t^{1-\alpha}\kappa\nabla u,\nabla u}\,dt
    \le\|u_0\|^2+\frac{C_\Omega\Psi(\alpha)}{\kappa_{\min}t^{1-\alpha}}
    \int_0^t\|g(s)\|^2\,ds
\quad\text{for $0\le t\le T$.}
\]
The proof follows the same lines as above, except that now
\[
\int_0^T\iprod{\partial_tu,u}\,dt
    +\int_0^T\bigiprod{\partial_t^{1-\alpha}\kappa\nabla u,\nabla u}\,dt
    =\int_0^T\iprod{g,u}\,dt
\]
and 
\[
\int_0^T\iprod{\partial_tu,u}\,dt=\tfrac12\bigl(\|u(T)\|^2-\|u_0\|^2\bigr).
\]
\end{remark}

\begin{remark}
Le et al.~\cite{LeMcLeanMustapha2016} proved stability and convergence of the 
DG scheme with general~$\vec F$, but only for the lowest-order ($p_n=0$) case 
and with no spatial discretization.  Although the constants are bounded 
as~$\alpha\to1$, they blow up as~$\alpha\to1/2$ and thus the fractional exponent 
is restricted to the range~$1/2<\alpha\le1$. Huang et al.~\cite{HuangEtAl2020} 
proved similar results for a slightly modified scheme.
\end{remark}

\printbibliography
\end{document}